\documentclass[11pt]{article}
\usepackage{mathtools}
\usepackage{amsthm}
\usepackage{amsmath}
\usepackage{amssymb}
\usepackage{mathrsfs}
\usepackage[style=trad-alpha]{biblatex}
\usepackage{geometry}

\usepackage[shortlabels]{enumitem} %To enumerate (list) vertically by (i), (ii), etc.
\usepackage[mathscr]{euscript}
\usepackage{multicol}

\usepackage{parskip} % för att ändra styckesindelning från indrag till rad

\usepackage{tikz-cd} % for commutative diagrams

\usepackage{centernot}

\usepackage{color} % to write with color

\usepackage{todonotes}

\usepackage[T1]{fontenc} %fixing something(?)

 \newtheorem{theorem}{Theorem}[section] %To write a theorem
\theoremstyle{definition}
\newtheorem{definition}[theorem]{Definition} %Definition
\newtheorem{proposition}[theorem]{Proposition}
\newtheorem{corollary}[theorem]{Corollary}
\newtheorem{lemma}[theorem]{Lemma}
\newtheorem{example}[theorem]{Example} %Example
\numberwithin{equation}{subsection}

\theoremstyle{remark}
\newtheorem{remark}[theorem]{Remark} %Remark
\numberwithin{equation}{section}

\DeclareMathOperator{\id}{id}

\DeclareMathOperator{\Spec}{Spec}

\DeclareMathOperator{\Span}{span}

\DeclareMathOperator{\rk}{rk}
\DeclareMathOperator{\Hom}{Hom}

\DeclareMathOperator{\Char}{char}

\DeclareMathOperator{\diag}{diag} % diagonal matrices
\DeclareMathOperator{\scalar}{scalar} % scalar matrices
\DeclareMathOperator{\aff}{{\bf Aff}} % for category of affine schemes
\DeclareMathOperator{\grpds}{{\bf Grpds}} % Category of groupoids
\DeclareMathOperator{\op}{op} % for opposite category
\DeclareMathOperator{\set}{{\bf Set}}
\DeclareMathOperator{\ef}{\text{ef}}
\DeclareMathOperator{\Ss}{\text{ss}}
\DeclareMathOperator{\rep}{\text{Rep}}

\DeclareMathOperator{\vect}{Vec}
\newcommand{\gr}{\text{Gr}}

\newcommand{\calo}{\mathcal{O}}

\newcommand{\cale}{\mathcal{E}}

\newcommand{\calm}{\mathcal{M}}

\newcommand{\ma}{\mathbb{A}}
\newcommand{\mz}{\mathbb{Z}}
\newcommand{\mq}{\mathbb{Q}}

\newcommand{\mg}{\mathbb{G}}

\usepackage{hyperref}
\hypersetup{
    colorlinks=true,
    linkcolor=blue,
    filecolor=blue,      
    urlcolor=blue,
}
 
\title{\textbf{Essentially finite $G$-torsors}}
\author{Archia Ghiasabadi\thanks{archia.ghiasabadi@gmail.com}, Stefan Reppen\thanks{stefan.reppen@gmail.com}}
\date{}

\begin{document}
\maketitle
\begin{abstract}
Let $X$ be a smooth projective curve of genus $g$, defined over an algebraically closed field $k$, and let $G$ be a connected reductive group over $k$. We say that a $G$-torsor is essentially finite if it admits a reduction to a finite group, generalising the notion of essentially finite vector bundles to arbitrary groups $G$. We give a Tannakian interpretation of such torsors, and we prove that all essentially finite $G$-torsors have torsion degree, and that the degree is 0 if $X$ is an elliptic curve. We then study the density of the set of $k$-points of essentially finite $G$-torsors of degree $0$, denoted $M_{G}^{\ef,0}$, inside $M_{G}^{\Ss,0}$, the $k$-points of all semistable degree 0 $G$-torsors. We show that when $g=1$, $M_{G}^{\ef}\subset M_{G}^{\Ss,0}$ is dense. When $g>1$ and when $\Char(k)=0$, we show that for any reductive group of semisimple rank 1, $M_{G}^{\ef,0}\subset M_{G}^{\Ss,0}$ is not dense. %As a corollary we see that $M_{GL_n}^{\ef,0}$ is not dense in $M_{GL_n}^{\Ss,0}\setminus M_{GL_n}^{s,0}$ for any $n>1$, where $M_{GL_n}^{s,0}$ denotes the moduli space of stable vector bundles of degree 0.
\end{abstract}
\tableofcontents
\noindent
\section{Introduction}
Let $X$ be a smooth projective connected curve over an algebraically closed field $k$. Let $g=g(X)$ be the genus of $X$. In 1938 Weil introduced the notion of a finite vector bundle; a vector bundle $E$ is called finite if there are two
distinct polynomials, $f,g\in \mathbb{N}[x]$, such that the
vector bundle $f(E)$ is isomorphic to $g(E)$ (see \cite{Weil1}). 
For $k=\mathbb C$, he proved that a vector bundle is finite if and only if it arises from a representation of $\pi_1(X)$ which factors through a finite group. Almost 40 years later, in \cite{Nori1}, Nori introduced the notion of an essentially finite vector bundle as a subquotient of a finite one. The category of essentially finite vector bundles forms a Tannakian category, and the corresponding group is known as the Nori fundamental group, a pro-group scheme over $k$ whose $k$ points are isomorphic to the étale fundamental group, $\pi_{1}^{\text{et}}(X)$, when $k$ is of characteristic 0 (see \cite[Corollary 6.7.20]{TAMAS} and also e.g., \cite{esnault}). 

Viewing a vector bundle as a $\text{GL}_n$-torsor, we are led to the question: can we generalise the notion of an essentially finite vector bundle, to a notion of an essentially finite $G$-torsor, for $G$ an affine algebraic group? Nori proved that a vector bundle $E$ is essentially finite if and only if there exists a finite group scheme $\Gamma$, a $\Gamma$-torsor $F_\Gamma$ and a representation $V$ of $\Gamma$ such that $E \cong F_\Gamma\times^{\Gamma}V$. Hence, we are led to the following definition 
\begin{definition}
An {\bf essentially finite} $G$-torsor is a $G$-torsor over $X$ which admits a reduction to a finite group.
\end{definition}
%say that a $G$-torsor is essentially finite if it admits a reduction to a finite subgroup.%
Under the correspondence between vector bundles and $\text{GL}_n$-torsors, this agrees with the known definition of essentially finite vector bundles. We prove the following.
\begin{theorem}
Let $G$ be a connected, reductive group. %There exists an $N\in \mathbb{N}$ such that, if either $\Char(k)=0$ or if $\Char(k)\geq N$, t
Then for any $G$-bundle $F_G$, the following are equivalent.
\begin{enumerate}
    \item The $G$-bundle $F_G$ is essentially finite.
    \item There exists a faithful representation $\rho\colon G\to \text{GL}_V$ such that $\rho_{*}F_G$ is an essentally finite vector bundle.
    \item For every representation $\rho\colon G\to \text{GL}_V$, $\rho_{*}F_G$ is an essentally finite vector bundle.
    \item There exists a proper surjective morphism $f\colon Y \to X$ such that $f^{*}F_G$ is trivial.
\end{enumerate}
\end{theorem}
 Note also that since semistability can be checked on the adjoint bundle, every essentially finite $G$-torsor is semistable. We give a self-contained proof of this fact, not using the adjoint representation.% prove the following generalisation of this fact.
%\begin{theorem}
%For any connected reductive group $G$, every essentially finite $G$-torsor over $X$ is semistable. If $\Char(k)=p>0$, then every essentially finite $G$-torsor over $X$ is strongly semistable.
%\end{theorem}
%For any connected reductive group $G$, we show that an essentially finite $G$-torsor is semistable in all characteristics and strongly semistable in characteristic $p>0$. 

Let now $M_{G}^{\Ss}$ denote the moduli space of semistable $G$-bundles over $X$, for $G$ a connected reductive group. Recall that the connected components of $M_{G}^{\Ss}$ are indexed by the algebraic fundamental group of $G$, $\pi_1(G)$. If a $G$-bundle, $F_G$, lies in a component corresponding to $d\in \pi_1(G)$, then it is said to have degree $d$. Essentially finite vector bundles always have degree 0. We prove the following.
\begin{theorem}
For any connected reductive group $G$, every essentially finite $G$-torsor over $X$ is of torsion degree.
\end{theorem}
%Furthermore, we show that essentially finite $G$-torsors have torsion degree, 
Again this generalises the case for $G=\text{GL}_n$, since in this case $\pi_1(G)=\mathbb{Z}$, which is torsion-free. We also show that if $X$ is an elliptic curve 
%or if $G$ is semisimple of type A, 
then all essentially finite $G$-bundles have degree 0.
%We then study the $k$-points of the essentially finite $G$-torsors of degree 0, denoted $M_{G}^{\ef,0}$, inside %the moduli space of semistable $G$-bundles of degree zero,
%$M_{G}^{\Ss,0}$. We are led by the following.
%
%Essentially finite vector bundles are semistable and of degree 0. Hence, if     $\mathcal{U}^{\Ss}(n,0)$ denotes the coarse moduli space of semistable vector bundles over $X$ of rank $n$ and degree 0, then  we can consider the subset $\mathcal{U}^{\ef}(k)\subset \mathcal{U}^{\Ss}(n,0)(k)$ of essentially finite vector bundles. 

Let now $M_{G}^{\ef,0}$ denote the $k$-points of the essentially finite $G$-torsors of degree 0, inside $M_{G}^{\Ss,0}$, and let $G=\text{GL}_n$. If $n=1$, then essentially finite $G$-bundles correspond to essentially finite line bundles, which correspond to torsion line bundles (see Lemma 3.1 \cite{Nori1}). Hence, $M_{\text{GL}_1}^{\ef}$ is dense inside $M_{\text{GL}_1}^{\Ss,0}=\text{Jac}^0(X)$ since torsion points are dense in any abelian variety. %In particular essentialy line bundle are dense inside $\text{Jac}^0(X)$.\\
In positive characteristic %much stronger results have been obtained. %If  $\mathcal{U}_{\calo_X}^{\Ss}(n)$ denotes the subspace of $\mathcal{U}^{\Ss}(n,0)$ consisting of vector bundles with trivial determinant, then
Ducrohet and Mehta have shown that $M_{\text{GL}_n}^{\ef,0} \subset M_{\text{GL}_n}^{\Ss,0}$ is dense for all $n$ when $g\geq 2$, and similarly for vector bundles with trivial determinant (they show in fact that a smaller set of objects, called Frobenius periodic vector bundles, are dense; see \cite{frobdensity}). However, in characteristic zero much less seems to be known about the density of essentially finite bundles when the rank is greater than 1. Hence, we may ask whether $M_{\text{GL}_n}^{\ef,0}$ is dense in $M_{\text{GL}_n}^{\Ss,0}$ for $n>1$, when $\Char(k)=0$. More generally, we are interested in the question of whether $M_{G}^{\ef,0}$ is dense in $M_{G}^{\Ss,0}$ for arbitrary connected reductive groups $G$ over an arbitrary, algebraically closed field $k$.

If $g=0$, that is if $X\cong \mathbb{P}^1$, then it is well-known that $M_{G}^{\Ss,0}(k)$ is a singleton. Hence it is clear that every essentially finite $G$-torsor over $\mathbb{P}^1$ is trivial. We give a self-contained proof of this result using a Tannakian interpretation of both the classification of $G$-torsors over $\mathbb{P}^1$ (see \cite{torsorsprojective}) and the definition of essentially finite torsors. If $g=1$, that is if $X$ is an elliptic curve, then we prove that $M_{G}^{\ef,0}$ is dense in $M_{G}^{\Ss,0}$ for all connected, reductive groups. This follows from work of Fr{\u{a}}{\c{t}}il{\u{a}} \cite{dragos} and Laszlo \cite{laszlo}. On the contrary, if $g\geq 2$ and $\Char(k)=0$, then we show the following.
\begin{theorem}
Let $\Char(k)=0$. For all connected, reductive groups of semisimple rank 1, $M_{G}^{\ef,0} \subset M_{G}^{\Ss,0}$ is not dense. 
\end{theorem}
%for all connected, reductive groups of semisimple rank 1,  $M_{G}^{\ef,0} \subset M_{G}^{\Ss,0}$ is not dense. 
The main work lies in proving the theorem for $\text{PGL}_2$-torsors. Note also that this shows that $M_{\text{GL}_2}^{\ef}$ is not dense in $M_{\text{GL}_2}^{\Ss,0}$. In characteristic 0, Weissman \cite{dario} has independently obtained this non-density result for $M_{\text{GL}_n}^{\ef}$ for all $n\geq 1$.
%
%As a corollary to the non-density for rank 2 vector bundles, we show that for any $n\geq 2$, the essentially finite vector bundles of rank $n$ are not dense in $M_{GL_n}^{\Ss,0}\setminus  M_{GL_n}^{s,0}$. The same holds also for vector bundles with trivial determinant.

By the theorem of Narasimhan and Seshadri, the points of 
$M_{\text{GL}_n}^{\Ss,0}(\mathbb{C})$  are also the isomorphism classes of representations $\pi_1(X)\longrightarrow \text{U}_n(\mathbb{C})$, i.e., there is an analytic homeomorphism between $M_{\text{GL}_n}^{\Ss,0}(\mathbb{C})$ and the character variety $\Hom(\pi_1(X),\text{U}_n(\mathbb{C}))/\sim$. In particular finite vector bundles correspond to unitary representations of $\pi_1(X)$ which factor through finite groups. As the Zariski topology is coarser than the analytic topology we see as a corollary to non-density for rank n vector bundles that the set of rank $n$ unitary representations of $\pi_1(X)$ which factor through finite groups is not dense inside $\Hom(\pi_1(X),\text{U}_n(\mathbb{C}))/\sim$.
%\begin{corollary}
%The set of rank 2 unitary representation of $\pi_1(X)$ wich factor through finite groups is not dense inside $\Hom(\pi_1(X),SU_2(\mathbb{C}))/\sim$.
%\end{corollary}
%\begin{remark}
%Upon finalising this article we where told that by results of the forthcoming PhD thesis of Dario Wei{\ss}mann, non-density for essentially finite vector bundles in fact holds for all $n>1$. %He studies stability under pullbacks along finite covers, whereas we use a more group theoretic approach. 
%\end{remark}

The outline of the text is as follows. In Section \ref{section.background} we introduce the necessary notations and background. In Section \ref{section.ef} we define essentially finite $G$-torsors, generalising the notion of essentially finite vector bundles. We prove that such torsors are (strongly) semistable of torsion degree. Finally, in Section \ref{section.density} we prove the above mentioned statements about density of $M_{G}^{\ef,0}$ in $M_{G}^{\Ss,0}$.
\subsection{Acknowledgements}
We would like to thank our respective advisors, Carlo Gasbarri and Wushi Goldring, for their support during this project. We would also like to thank Drago{\c{s}} Fr{\u{a}}{\c{t}}il{\u{a}}, João Pedro dos Santos, Emiliano Ambrosi, Florent Schaffhauser, and Georgios Kydonakis for their remarks and questions that have greatly improved the quality of this work. Finally, we thank an anonymous referee for explaining to us that we could relax the assumptions in Proposition \ref{nori.ef.equiv} and how this gives a much simpler proof of the implication ``3. implies 1.'' in Theorem \ref{ef.via.reps} (which also allowed us to remove a restriction on the characteristic in an earlier version), as well as for several other useful comments that greatly improved the quality of the paper.
\section{Notations, conventions and background}\label{section.background}
Throughout the text, let $k$ be an algebraically closed field and let $X$ be a smooth, projective, connected curve over $k$. Recall that if $G$ denotes an algebraic group over $k$, then a $G$-torsor over $X$ is a scheme $F_G$ over $X$ with an action of $G$ such that there exists an fppf cover, $(U_i \to X)_{i \in I}$ such that for each $i\in I$ there is a $G|_{U_i}$-equivariant isomorphism  $F_G|_{U_i} \cong G|_{U_i}$. We will also use the term $G$-bundle as synonym for $G$-torsor. 
%Somewhat non-standard, we will denote by $1_G$ the trivial torsor $1_G \coloneqq X\times G$.
If $\varphi : H\to G$ is a group morphism and $F_H$ is an $H$-torsor, then we denote by $\varphi_{*}F_H$ the $G$-torsor $\varphi_{*}F_H\coloneqq F_H \times^{H} G$. In the special case when $\varphi : G \to \text{GL}_V$ is a representation of $G$, we denote $\varphi_{*}F_G$ by $V_{F_G}$ (following \cite{Schieder}). If $F_G$ is a $G$-torsor such that $F_G \cong \varphi_{*}F_H$ for some triple $(H,\varphi, F_H)$ as above, then we say that $F_G$ admits a reduction of structure group to $H$. % If $Y$ is a scheme then we denote by $\calm_{G,Y}$ the stack of $G$-torsors over $Y$ and, when it exists, we denote by $\calm_{G,Y}^{\Ss}$ the moduli space of semistable $G$-torsors over $Y$. When $Y=X$, we write $\calm_{G}=\calm_{G,X}$ and similarly $\calm_{G}^{\Ss}=\calm_{G,X}^{\Ss}$.%Let now $G$ be a linear algebraic group. 
We denote by $\rep_k(G)$ the category of finite-dimensional representations of $G$ over $k$. Recall that to give a $G$-torsor over $X$ is equivalent to give an exact, $k$-linear, tensor functor $\rep_k(G)\to \vect_X$, where $\vect_X$ denotes the category of vector bundles over $X$. We will use the same notation for the bundle seen as a functor.

Now suppose that $G$ is a connected, reductive group. Given a  maximal torus $T\subset G$ let $X^{*}(T)$ denote the characters of $T$ and let $X_{*}(T)$ denote the cocharacters. Let further $\Phi\subset X^{*}(T)$ denote the corresponding roots and let $\Phi^{\vee}\subset X_{*}(T)$ denote the corresponding coroots. We let $\pi_1(G)$ denote the algebraic fundamental group of $G$, namely, 
\begin{equation}
    \pi_1(G) = X_{*}(T)/\Span\{\Phi^{\vee}\}.
\end{equation}
Given a parabolic $P\subset G$ with Levi quotient $L$, let $\Phi_{L}^{\vee} \subset \Phi^{\vee}$ denote the coroots of $L$. We write $\pi_1(P) \coloneqq \pi_1(L)$. 

Let $\calm_G$ denote the stack of $G$-torsors over $X$, let $\calm_{G}^{\Ss}$ denote the substack of semistable $G$-torsors and let $M_{G}^{\Ss}$ denote the moduli space of semistable $G$-torsors (see \cite{ramana1}, \cite{ramana2} and \cite{modulispace.arb.char}). If we consider another curve, $Y$, then for clarity we may also write $\calm_{G,Y}$ to denote the stack of $G$-torsors over $Y$. We define $\calm_{G,Y}^{\Ss}$ and $M_{G,Y}^{\Ss}$ analogously.

Recall that the connected components of $\calm_{G}$ are labeled by $\pi_1(G)$, that is,
\begin{equation}
    \pi_0(\calm_{G}) = \pi_1(G).
\end{equation}
If $\check{\lambda} \in \pi_1(G)$, let $\calm_{G}^{\check{\lambda}}\subset \calm_{G}$ denote the corresponding component. Define similarly $\calm_{G}^{\Ss,\check{\lambda}}$ and $M_{G}^{\Ss,\check{\lambda}}$ to be the components in $\calm_{G}^{\Ss}$ respectively $M_{G}^{\Ss}$ corresponding to $\check{\lambda}$.
\begin{definition}
If $F_G$ is an object of $\calm_{G}^{\check{\lambda}}$, then $F_G$ is said to be of {\bf degree} $\check{\lambda}$.
\end{definition}
We also have that $\pi_0(\calm_P)\cong \pi_0(\calm_L) = \pi_1(P)$ and we similarly say that a $P$-torsor is of degree $\check{\lambda}_P$ if it lies in the component corresponding to $\check{\lambda}_P$.
\begin{lemma}\label{deg0push}
Suppose that $\varphi : G\to H$ is a morphism of smooth connected algebraic groups and let $F_G$ be a $G$-torsor of degree 0. Then $\varphi_{*}F_{G}$ has degree 0.
\end{lemma}
\begin{proof}
By \cite{hoffman} we have a commutative diagram of pointed sets
\begin{equation}
    \begin{tikzcd}
    \pi_1(G) \arrow[r] \arrow[d] & \pi_0(\calm_G) \arrow[d] \\
    \pi_1(H) \arrow[r] & \pi_0(\calm_H),
    \end{tikzcd}
\end{equation}
where all morphisms are the natural ones induced by $\varphi$ and where the left vertical map is a group morphism. The statement follows.
\end{proof}
\begin{remark}
In particular, if $F_G$ is a $G$-bundle of degree 0 then $\deg V_{F_G}=0$ for all representations $V$ of $G$.
\end{remark}
%\section{Torsors}
%Let $X$ be a smooth, projective curve over a field $k$, and let $G$ be an algebraic group over $k$. Recall that a {\bf $G$-torsor over $X$} is a scheme $F_G$ over $X$ with an action of $G$ such that there exists an fppf cover, $(U_i \to X)_{i \in I}$ such that for each $i\in I$ there is a $G|_{U_i}$-equivariant isomorphism  $F_G|_{U_i} \cong G|_{U_i}$. If $\calg$ is a sheaf of groups and $\calf$ is a sheaf on $X$, then we say that $\calf$ is a {\bf sheaf $\calg$-torsor} if there is an fppf cover $(U_i \to X)$ of $X$ such that for each $i$, $\calf(U_i)$ is a $\calg(U_i)$-set and there exists $\mathcal{G}|_{U_i}$-equivariant isomorphisms $\mathcal{F}|_{U_i} \cong \mathcal{G}|_{U_i}$. Via the Yoneda embedding we can also view $G$ as a sheaf of groups on $X$, i.e. $U \mapsto \Hom(U,G)$. Since all the groups we will consider are affine every sheaf torsor is representable by a torsor so we will use the term $G$-torsor interchangably for both torsors and sheaf torsors.
%
%
%
% GENERALITIES ON TORSORS
%
%
%
%
%
%
% SEMISTABLE TORSORS
%
\subsection{Semistable torsors}
%Let now $k$ be algebraically closed of any characteristic and let $G$ be a reductive group over $k$. 
Let $T$ be a maximal torus of $G$ and let $B\supset T$ be a Borel containing $T$. Then the center of $G$ can be described as 
\begin{equation}
    Z(G)=\bigcap_{\alpha\in \Phi} \ker(\alpha) \subset T.
\end{equation}
By composition via the inclusion $Z(G)\to T$ we have a natural map
\begin{equation}
    X_{*}(Z(G))\to X_{*}(T) \to \pi_1(G).
\end{equation}
%where $X_{*}(H)$ denotes the cocharacters of a group $H$.
Upon tensoring with $\mq$ this induces an isomorphism $X_{*}(Z(G))_{\mq}\cong \pi_1(G)_{\mq}$. 
Following \cite{Schieder} the definition of the slope map and subsequently the definition of a semistable $G$-torsor is as follows.
%%%% DEFINITION OF SLOPE MAP
\begin{definition}
For a parabolic subgroup, $P$, such that $B\subset P \subset G$, with corresponding Levi $L$, the {\bf slope map} $\phi_P : \pi_1(P)\to X_{*}(T)_{\mq}$ is the map given by
\begin{equation}\label{phi.P}
    \phi_P : \pi_1(P)\to \pi_1(P)_{\mq}\cong X_{*}(Z(L))_{\mq} \to X_{*}(T)_{\mq}.
\end{equation}
\end{definition}
%%%%%% EXAMPLE OF SLOPE MAP
\begin{example}
For $G=\text{GL}_n$, we will describe the slope map $\phi_G$. We have that $L=G$ so $Z(L)=\scalar_n$, the scalar matrices of rank $n$. We also have the standard identifications $X_{*}(\scalar_n)\cong \mz$ and $X_{*}(T)\cong \mz^{n}$. Further, we may write $\pi_1(G)= \mz \cdot \overline{e}_1$, where $e_i : t \mapsto \diag(1,...,1,t,1,...,1)$ with $t$ in the $i^{\text{th}}$ position, and $\overline{(-)}$ represents the image in $\pi_1(G)$. Then we have that $\overline{(a,...,a)}= na \overline{e}_1$, hence the morphism $X_{*}(\scalar_n)\to \pi_1(G)$ is simply
\begin{equation}
\begin{aligned}
    X_{*}(\scalar_n) \to X_{*}(T) \to \pi_1(G) = \mz \overline{e}_1 \\
    a \mapsto (a,...,a) \mapsto \overline{(a,...,a)} = na \overline{e}_1,
\end{aligned}
\end{equation}
i.e., multiplication by $n$. Thus, upon tensoring with $\mq$ the morphism $\phi_G$ from (\ref{phi.P}) is given by
\begin{equation}
    \begin{aligned}
       \pi_1(G)&\to \pi_1(G)_{\mq} \xrightarrow{\cong} X_{*}(\scalar_n)_{\mq} \to X_{*}(T)\\
        a &\mapsto \frac{a}{1} \mapsto \frac{a}{n} \mapsto (\frac{a}{n}, ..., \frac{a}{n}).
    \end{aligned}
\end{equation}
%Thus, if $F_{G}$ is a vector bundle of degree $d \in \pi_1(G)$, we see that $\phi_G(d)=(\frac{d}{n},..., \frac{d}{n})$.

Now let $P$ be an arbitrary parabolic of $G=\text{GL}_n$, with Levi factor $L = \prod_{i=1}^{m}\text{GL}_{n_i}$. Then $Z(L) = \prod_{i=1}^{m} \scalar_{n_i} \cong \mz^{m}$. The isomorphism $\pi_1(P)_\mq \to X_{*}(Z(L))_{\mq}$ is the inverse to the morphism
\begin{equation}
\begin{aligned}
    X_{*}(Z(L))\cong\mz^{m} &\to \mz^{n} \to \mz^{m}\cong\pi_1(P) \\
    (a_1, ..., a_m) &\mapsto (a_1, ..., a_1, a_2, ..., a_2, ..., a_m, ..., a_m) \mapsto (n_1 a_1, n_2 a_2, ..., n_m a_2),
\end{aligned}
\end{equation}
where $a_i$ occurs $n_i$ times in the tuple in the middle. Thus, the slope map $\phi_P$ is given by
\begin{equation}
\begin{aligned}
    \pi_1(P)&\to \pi_1(P)_{\mq}\cong X_{*}(Z(L))_{\mq} \to X_{*}(T)_{\mq} \\
    (a_1, ..., a_m) &\mapsto (\frac{a_1}{1}, ..., \frac{a_m}{1}) \mapsto (\frac{a_1}{n_1}, ..., \frac{a_m}{n_m}) \mapsto (\frac{a_1}{n_1},..., \frac{a_1}{n_1}, ..., \frac{a_m}{n_m}, ..., \frac{a_m}{n_m}).
\end{aligned}
\end{equation}
\end{example}
%%%%%% DEFINITION OF SEMI-STABLE
\begin{definition}\label{def.ss}
Let $F_G$ be a $G$–torsor of degree $\check{\lambda}$. We say that $F_G$ is {\bf semi-stable} if for each parabolic $P\subset G$ and each reduction $F_P$ of $F_G$ to $P$, of degree $\check{\lambda}_P$, we have that
\begin{equation}
    \phi_P(\check{\lambda}_P) \leq \phi_G(\check{\lambda}).
\end{equation}
\end{definition}
\begin{remark}
If $\phi_P(\check{\lambda}_P) < \phi_G(\check{\lambda})$ then $F_G$ is called {\bf stable}.
\end{remark}
\begin{example}
Again let $G=\text{GL}_n$, we show why this definition gives back the usual slope semi-stability for vector bundles. Recall first that the slope $\mu(E)$ of a vector bundle $E$ is defined as $\mu(E) = \frac{deg(E)}{rk(E)}$ and that $E$ is called slope semi-stable if for any subbundle $F$ we have that $\mu(F)\leq \mu(E)$. 

Let now $E$ be a vector bundle, let $P \subset G$ be a parabolic with Levi factor $L=\prod_{i=1}^{m}\text{GL}_{n_i}$ and let $F_P$ be a reduction of $E$ to $P$. This amounts to giving a filtration $0 \subset E_1 \subset ... \subset E_m = E$, where $\rk E_i - \rk E_{i-1} = n_i$. Then $\deg(F_P) = (\deg(\pi_{1,*}F_P), ..., \deg(\pi_{m,*}F_P))$ where $\pi_i : P \to L \to \text{GL}_{n_i}$ is the composition of the projections $P\to L$ and $L \to \text{GL}_{n_i}$. Then we see that
\begin{equation}
\begin{aligned}
    \phi_P(\deg(F_P)) &= (\frac{\deg(\pi_{1,*}F_P)}{n_1},..., \frac{\deg(\pi_{1,*}F_P)}{n_1}, ..., \frac{\deg(\pi_{m,*}F_P)}{n_m}, ..., \frac{\deg(\pi_{m,*}F_P)}{n_m}) \\
    &=(\mu(E_1), ..., \mu(E_1), ..., \mu(E_m/E_{m-1}), ..., \mu(E_m/E_{m-1})).
\end{aligned}
\end{equation}
Since $\phi_G(\deg(E))=(\mu(E), ..., \mu(E))$ we see that Definition \ref{def.ss} agrees with the usual slope semi-stability definition. 
\end{example}
%%%%%% SCHIEDER STUFF THAT WE WILL NEED TO PROVE EF IMPLIES SS
Now we recall some results of \cite{Schieder} regarding the slope map which we will need to prove that essentially finite torsors are semi-stable. To this end, let $\lambda \in X^{*}(T)$ be a dominant character and let $V$ be a finite-dimensional $G$-representation of highest weight $\lambda$. If $P$ is a parabolic with Levi factor $L$, and if $V=\bigoplus_{\nu\in X^{*}(T)}V[\nu]$ is the weight space-decomposition of $V$, then let
\begin{equation}
    V[\lambda + \mz \Phi_{L}] \coloneqq \bigoplus_{\nu \in \lambda + \mz \Phi_{L}} V[\nu],
\end{equation}
where $\Phi_L$ are the roots of the Levi $L$. Then we have the following result.
\begin{proposition}[\cite{Schieder} Proposition 3.2.5(b),(c)]\label{schieder.prop.slope}
Keep the notation as above. Let $F_G$ be a $G$-torsor of degree $\check{\lambda}_G$. Then the slope of the vector bundle $V_{F_G}$ is given by
\begin{equation}
    \mu(V_{F_G}) = \langle \phi_G(\check{\lambda}_G), \lambda \rangle.
\end{equation}
Furthermore, if $F_P$ is a $P$-torsor of degree $\check{\lambda}_P$ with corresponding Levi bundle $F_L$, then the vector bundle $V[\lambda + \mz \Phi_L]_{F_L}$ has slope
\begin{equation}
    \mu(V[\lambda + \mz \Phi_L]_{F_L}) = \langle \phi_P(\check{\lambda}_P), \lambda \rangle.
\end{equation}
\end{proposition}
%\begin{proposition}
%Let $F_G$ be a $G$-torsor and let $\check{\lambda}_G$ be its degree. If $\check{\lambda}_G=0$ %if and only if 
%then $\deg V_{F_G}=0$ for all representations $V$ of $G$.
%\end{proposition}
%\begin{proof}
%Suppose $\check{\lambda}_G=0$ and let $V$ be a representation of $G$ with weight space decomposition $V=\oplus_\nu V[\nu]$, and let $m_\nu$ denote the multiplicity of $V[\nu]$. Then by Proposition  3.2.5 of \cite{Schieder} we see that 
%\begin{equation}
%    \deg V_{F_G}= \langle \check{\lambda}_G , \sum_{m_\nu} m_\nu \nu \rangle = 0.
%\end{equation}
%Conversely, suppose $\deg V_{F_G}=0$ for all representation $V$ of $G$. Let $\lambda$ be any dominant character of $G$. Then we see again from Proposition 3.2.5 of \cite{Schieder} that
%\begin{equation}
%     0 = \deg V[\lambda]_{F_G} = \langle \check{\lambda}_G , \lambda \rangle.
%\end{equation}
%Hence $\check{\lambda}_G=0$.
%\end{proof}
%
%
%
%
%
% ESSENTIALLY FINITE TORSORS
%
\section{Essentially finite torsors}\label{section.ef}
We begin with the main object of study in this article.
\begin{definition}
An {\bf essentially finite} $G$-torsor is a $G$-torsor over $X$ which admits a reduction to a finite group.
\end{definition}
\begin{remark}
Although we have fixed a smooth, projective, connected curve $X$ over $k$ for simplicity of the exposition, this definition makes sense over an arbitrary scheme. Similarly, we may use the same definition for arbitrary affine groups, not necessarily connected reductive.
\end{remark}
\begin{remark}
Note that if $\varphi : \Gamma \to G$ is a map from a finite group $\Gamma$, then we obtain an injection $\tilde{\varphi} : \Gamma/\ker(\varphi) \hookrightarrow G$. If $F_\Gamma$ is a $\Gamma$-torsor, then $\varphi_{*}F_\Gamma=\tilde{\varphi}_{*}(\pi_{*}F_\Gamma)$ as $G$–torsors, so we can always assume $\Gamma$ to be a subgroup of $G$.
\end{remark}
%\begin{remark}
%For simplicity, we shall often write $(F_G, F_\Gamma, \iota)$ for an essentially finite $G$-torsor, $F_G$, where $\Gamma$ is a finite group, $\iota : \Gamma\hookrightarrow G$ and $F_\Gamma$ is a $\Gamma$-torsor such that $\iota_{*}F_\Gamma\cong F_G$. But remark that in the definition for an essentially finite $G$-torsor, the specific group $\Gamma$ is not included.
%\end{remark}
%
\begin{example}
\begin{enumerate}
    \item The trivial $G$-torsor $G\times X$ is essentially finite since it admits a reduction to the trivial group.
    \item If $\Gamma$ is finite then every $\Gamma$-torsor $F_\Gamma$ is essentially finite since $F_\Gamma\cong \id_{*}F_\Gamma$.
    \item Note that if $\alpha: G\to G'$ is a morphism of algebraic groups and $F_G$ is an essentially finite $G$-torsor, then $\alpha_{*}F_G$ is an essentially finite $G'$-torsor.
\end{enumerate}
\end{example}
%
% CORRESPONDENCE with Nori reps into finite groups
Let us phrase two equivalent conditions for a $G$-bundle to be essentially finite; one in terms of the Nori fundamental group, and one Tannakian interpretation. Since $k$ is algebraically closed, there is a rational point $x$ of $X$. Let $\pi_{1}^{N}(X,x)$ denote the Nori fundamental group of $X$ and let $\widetilde{X}$ denote the universal $\pi_{1}^{N}(X,x)$-torsor over $X$, introduced in \cite{Nori1}.
\begin{proposition}\label{nori.ef.equiv}
A $G$-bundle $F_G$ is essentially finite if and only if there exists a morphism $\rho\colon \pi_{1}^{N}(X,x)\to G$ such that $\rho_{*}\widetilde{X}\cong F_G$. 
\end{proposition}
\begin{proof}
Let $F_G$ be an essentially finite $G$-torsor, let $\iota \colon \Gamma \hookrightarrow G$ be a finite subgroup of $G$ and let $j \colon F_{\Gamma}\to X$ be a $\Gamma$-torsor such that $\iota_{*}F_{\Gamma}\cong F_G$. Let $y$ be a rational point of $F_{\Gamma}$ such that $j(y)=x$. Then $j$ defines a pointed finite torsor $(F_{\Gamma},y)\to (X,x)$. By \cite[Proposition 3.11]{Nori1}, there is a morphism $\pi_{1}^{N}(X,x)\to \Gamma$, which we compose with $\iota$ to get a morphism $\rho\colon \pi_{1}^{N}(X,x) \to G$  such that $F_G\cong \rho_{*}\widetilde{X}$.

Conversely, suppose that we have a morphism $\rho\colon \pi_1^N(X,x)\to G$ 
%suppose we have a commutative diagram as below
%\begin{equation}
%\begin{tikzcd}
 %   \pi_{1}^{N}(X,x) \arrow[rr, bend left=30, "\rho"] \arrow[r, "\phi"] & \Gamma \arrow[r, "\iota"] & G
%\end{tikzcd},
%\end{equation}
such that $\rho_{*}\widetilde{X}\cong F_G$. Since $\pi_1^N(X,x)=\varprojlim_{i}A_i$ is the inverse limit of its finite quotients $A_i$ (see \cite{Nori2}), there is %\todo{add reference e.g. stacks} 
some $i$ and a morphism $\rho_i \colon A_i \to G$ such that $\rho$ factors 
\begin{equation}
    \rho\colon \pi_1^N(X,x)\xrightarrow{\pi_i} A_i \xrightarrow{\rho_i} G
\end{equation}
where $\pi_i$ is the projection. Since $\rho_{*}\widetilde{X}\cong \rho_{i,*}(\pi_{i,*}\widetilde{X})$ we see that $F_G$ is essentially finite.
%Then letting $\widetilde{X}$ denote the universal $\pi_{1}^{N}(X,x)$-torsor over $X$, we obtain an essentially finite $G$-torsor, $F_G\coloneqq \iota_{*}\rho_{*}\widetilde{X}$.
\end{proof}
\begin{proposition}\label{ef.tannaka}
A $G$-torsor $F_G$ is essentially finite if and only if there exists a finite group $\Gamma$, a $\Gamma$-torsor $F_\Gamma$, and a tensor functor $\alpha :\rep_k(G)\to \rep_k(\Gamma)$ such that : 
\begin{enumerate}
    \item we have that $\omega_\Gamma \circ \alpha = \omega_G$, where $\omega_G : \rep_k(G)\to \vect_k$ and $\omega_\Gamma : \rep_k(\Gamma)\to \vect_k$ are the forgetful functors; and
    \item we have a commutative diagram
    \begin{equation}
        \begin{tikzcd}
        \rep_k(G) \arrow[r, "F_G"] \arrow[d, "\alpha"'] & \text{Vec}_X \\
        \rep_k(\Gamma) \arrow[ru, "F_\Gamma"'] & 
        \end{tikzcd}
    \end{equation}
\end{enumerate}
\end{proposition}
\begin{proof}
If $F_G$ is essentially finite, coming from a finite group $\Gamma$, a group morphism $\varphi: \Gamma \to G$ and a $\Gamma$-torsor $F_\Gamma$, then we take $\alpha$ to be the induced functor from $\varphi$. Conversely, every such $\alpha$, by \cite[Corollary 2.9]{Tannaka}, comes from a group morphism $\varphi : \Gamma \to G$. 
\end{proof}
%By using the work of Biswas \todo{add ref + authors} we have the following.
\begin{remark}
If a $G$-torsor $F_G$ is essentially finite then there exists a finite group $\Gamma$ and a $\Gamma$-torsor $j_\Gamma : F_\Gamma \to X$ such that $j_{\Gamma}^{*}F_G$ is trivial.
\end{remark}
%\begin{proof}
%If $F_G$ is essentially finite, then we may choose $j_\Gamma : F_\Gamma \to X$ to be a torsor such that $F_G = \iota_{*}\Gamma$, for some $\iota : \Gamma \to G$. Conversely, since $k$ is algebraically closed, for any algebraic group $H$ and any $H$-torsor $F_H\to X$, to any point $x\in X(k)$ there is a $k$-point of $F_H$ lying over $x$. Hence, the result follows from \todo{ref to Biswas/ref to explanation of Biswas above}.
%\end{proof}
%\begin{proposition}\label{trivialized}
%Let $F_G$ be an essentially finite $G$-torsor. Let $\iota : \Gamma \hookrightarrow G$ be a finite subgroup of $G$ and $f_\Gamma : F_\Gamma \to X$ be a $\Gamma$-torsor such that $F_G \cong \iota_{*}F_\Gamma$. Then $f_{\Gamma}^{*}F_G\cong 1_G$.
%\end{proposition}
%\begin{proof}
%The statement follows from the follow diagram, where the square is Cartesian,
%\begin{equation}
%\begin{tikzcd}
%F_\Gamma \arrow[ddr, bend right= 30, "\id"'] \arrow[drr, bend left= 30] \arrow[dr, dotted] & &  \\
% & f_{\Gamma}^{*}F_G \arrow[r] \arrow[d] & F_G=\iota_{*}F_\Gamma \arrow[d] \\
% & F_\Gamma \arrow[r] & X
%\end{tikzcd}.
%\end{equation}
%So we have a section of the pullback $f_{\Gamma}^{*}F_G$, hence it is trivial.
%\end{proof}
%
%
%
%
%
% ESS FINITE GLn Bundles
\begin{proposition}\label{fintorsor.finvector}
Under the correspondence between vector bundles of rank $n$ and $\text{GL}_n$-torsors, a $\text{GL}_{n}$-torsor is essentially finite if and only if the corresponding vector bundle is essentially finite.
\end{proposition}
\begin{proof}
Let $F_{\text{GL}_n}$ be a $\text{GL}_{n}$-torsor, and let $\Gamma$ be a finite subgroup of $\text{GL}_n$, $\alpha : \Gamma \to \text{GL}_{n}$ and let $j : F_\Gamma\to X$ be a $\Gamma$-torsor such that $F_{\text{GL}_n} = \alpha_{*}F_\Gamma$. Then  $F_{\text{GL}_n}$ is trivialised by $j : F_\Gamma \to X$ so the corresponding vector bundle $E$ is also trivialised by $j : F_\Gamma\to X$. Thus, $E$ is essentially finite.
%
%By the proposition above we know that $j^{*}F_{GL_n}$ is the trivial $GL_n$-torsor. Then the corresponding vector bundle $E$ is trivialised by $j : F_\Gamma\to X$, so $E$ is essentially finite.
%
%Then the corresponding vector bundle is given by
%\begin{equation}
%E = P\times^{GL_n} \ma^n \cong Q \times^{\Gamma} GL_n \times^{GL_n} \ma^n \cong Q\times^{\Gamma} \ma^n.
%\end{equation}
%By definition (\textcolor{red}{add in INTRO/or Prop 3.11}) we see that $E$ is essentially finite.

Conversely suppose $E$ is an essentially finite vector bundle. Then there is a finite group $\iota : \Gamma\to \text{GL}_n$ and a $\Gamma$-torsor $F_\Gamma\to X$ such that $E=F_\Gamma\times^{\Gamma}\ma^n$. Then we have that
\begin{equation}
    E = F_\Gamma\times^{\Gamma}\ma^n \cong F_\Gamma\times^{\Gamma} \text{GL}_n \times^{\text{GL}_n} \ma^n \cong \iota_{*}F_\Gamma\times^{\text{GL}_n} \ma^n,
\end{equation}
whence the vector bundle associated to $\iota_{*}F_\Gamma$ is $E$. Hence, the bundle corresponding to $E$ is isomorphic to $\iota_{*}F_\Gamma$, hence essentially finite. 
\end{proof}
\begin{lemma}\label{triviality.via.faithful.reps}
Let $Y$ be a proper and connected scheme over $k$. A $G$-bundle $F_G$ over $Y$ is trivial if and only if for any faithful representation $\rho\colon G\to GL_V$, $\rho_*F_G$ is trivial.
\end{lemma}
\begin{proof}
The idea of this can be found in \cite[Lemma 4.5]{bisdos.via.referee}, but we spell out the details since their assumptions on the base scheme are different from ours. Suppose that $\rho\colon G\to GL_V$ is any faithful representation. Consider the long exact sequence of pointed sets (see \cite[III, \S 4, 4.6]{demazure.gabriel})
\begin{equation}\label{subgroups.les.h1}
    1\to G(Y)\xrightarrow{\rho} GL_V(Y) \xrightarrow{\pi} (GL_V/G)(Y) \xrightarrow{\delta} H^1(Y,G)\xrightarrow{\rho_*} H^1(Y,GL_V),
\end{equation}
where $\pi\colon GL_V\to GL_V/G$ is the canonical projection. 
The morphism $\delta$ takes a $Y$-point $y\colon Y\to GL_V/G$ to the $G$-bundle $\delta(y)\coloneqq Y\times_{GL_V/G,y,\pi}GL_V$. Since $G$ is reductive, $GL_V/G$ is affine and hence, using that $Y$ is proper and connected, $y$ is constant. That is, we have a factorisation $y\colon Y\to \Spec k \to GL_V/G$. Since $k$ is algebraically closed, $(GL_V/G)(k)=GL_V(k)/G(k)$, and hence $y$ being constant implies that there is a lift $\tilde{y}\colon Y\to GL_V$ of $y$. By the universal propery of fiber products we thus see that $\delta(y)$ admits a section, whence $\delta(y)$ is trivial. Hence, by exactness of \eqref{subgroups.les.h1} a $G$-bundle $F_G$ is trivial if and only if $\rho_*F_G$ is trivial.  
\end{proof}
%
%
%
%
% EF BUNDLES VIA REPRESENTATIONS
\begin{theorem}\label{ef.via.reps}
Let $G$ be a connected, reductive group and let $F_G$ be a $G$-bundle. %There exists an $N\in \mathbb{N}$ such that, if either $\Char(k)=0$ or if $\Char(k)\geq N$, then for any $G$-bundle $F_G$, t
The following are equivalent.
\begin{enumerate}
    \item The $G$-bundle $F_G$ is essentially finite.
    \item There exists a faithful representation $\rho\colon G\to \text{GL}_V$ such that $\rho_{*}F_G$ is an essentally finite vector bundle.
    \item For every representation $\rho\colon G\to \text{GL}_V$, $\rho_{*}F_G$ is an essentally finite vector bundle.
    \item There exists a proper surjective morphism $f\colon Y \to X$ such that $f^{*}F_G$ is trivial.
\end{enumerate}
\end{theorem}
\begin{proof}
By above we see that 1. implies 3., and it is clear that 3. implies 2.  %We also see that 2. implies 3. since any faithful representation is a tensor generator of $\rep_k(G)$.
By \cite{bds} 4. is equivalent to 3. Hence we prove that 2. implies 3. and that 3. implies 1.

First suppose that 2. holds, let $\varphi\colon G\to GL_W$ be a faithful representation such that $\varphi_*F_G$ is essentially finite and let $\rho\colon G\to GL_V$ be an arbitrary representation. Since $\varphi_*F_G$ is essentially finite %, by \cite{bds} 
there is a proper surjective morphism $f\colon Y\to X$ such that $f^*\varphi_*F_G$ is trivial. Since any restriction of $f^*\varphi_* F_G$ to a connected component of $Y$ is trivial, we may assume that $Y$ is connected.  %Further, since $Y_{\text{red}}\to Y$ is proper, we may assume that $Y$ is reduced. 
Thus, since $f^*\varphi_*F_G\cong \varphi_*f^*F_G$, we see from Lemma \ref{triviality.via.faithful.reps} that $f^*F_G$ is trivial. %Indeed, 
%since $G$ is reductive, $GL_W/G$ is affine\todo{add ref}. As $X$ is projective and $f$ is finite, $Y$ is projective, whence $(GL_W/G)(Y)$ consists only of constant morphisms. Consider now the long exact sequence of pointed sets
%\begin{equation}\label{subgroups.les.h1}
%    1\to G\xrightarrow{\varphi} GL_V \xrightarrow{\pi} GL_W/G \xrightarrow{\delta} H^1(Y,G)\xrightarrow{\varphi_*} H^1(Y,GL_V),
%\end{equation}
%where $\pi\colon GL_W\to GL_W/G$ is the canonical projection. 
%The morphism $\delta$ takes a $Y$-point $y\colon Y\to GL_W/G$ to the $G$-bundle $\delta(y)\coloneqq Y\times_{GL_W/G,y,\pi}GL_W$\todo{add ref}. Since $y$ is constant, there is a lift $\tilde{y}\colon Y\to GL_W$. By the universal propery of fiber products we thus see that $\delta(y)$ admits a section, that is, $\delta(y)$ is trivial. Hence, by exactness of \eqref{subgroups.les.h1} a $G$-bundle $E_G$ is trivial if and only if $\varphi_*E_G$ is trivial. In particular, $f^*F_G$ is trivial. 
Hence, $f^*\rho_*F_G \cong \rho_* f^*F_G$ is trivial, which implies that $\rho_* F_G$ is essentially finite (again by \cite{bds}). This proves that 2. implies 3.

Now assume that 3. holds. Then the functor $F_G\colon \rep_k(G)\to \vect_X$ factors through the category of essentially finite vector bundles, hence induces a group morphism $\rho\colon \pi_1^N(X,x)\to G$ such that $\rho_*\widetilde{X}\cong F_G$. Thus, by Proposition \ref{nori.ef.equiv} $F_G$ is essentially finite.
\end{proof}
%%
%
%
%
% DEGREE of GBUndles and SEMISTABILITY
\begin{proposition}
Every essentially finite $G$-torsor is semistable. %Let $F_G$ be an essentially finite $G$-torsor. Then 
%$F_G$ is semi-stable.
\end{proposition}
\begin{proof}
Let $F_G$ be such a torsor. Let further $P\subset G$ be a parabolic of $G$, let $\lambda$ be a dominant character and let $V$ be a representation of highest weight $\lambda$. Since $F_G$ is essentially finite, the associated vector bundle $V_{F_G}$ is essentially finite, hence  semistable. Hence, using Proposition \ref{schieder.prop.slope}, we have that
\begin{equation}
    \langle \psi_G(\check{\lambda}_G), \lambda \rangle = \mu(V_{F_G}) \geq \mu(V[\lambda + \mathbb{Z}\Phi_L]_{F_L}) = \langle \psi_P(\check{\lambda}_P) , \lambda \rangle.
\end{equation}
That is, for every dominant character $\lambda \in X^{*}(T)_{\mathbb{Q}}$ we have that
\begin{equation}\label{alpha.0}
    \langle \psi_G(\check{\lambda}_G)-\psi_P(\check{\lambda}_P)) , \lambda \rangle \geq 0.
\end{equation}
Since the cone of cocharacters with non-negative pairing with all dominant characters is double-dual to the cone of simple coroots, we see that 
\begin{equation}
    \psi_G(\check{\lambda}_G)-\psi_P(\check{\lambda}_P))\geq 0.
\end{equation}
\end{proof}
\begin{theorem}
Let $F_G$ be an essentially finite $G$-torsor. Then its degree is torsion as an element of $\pi_1(G)$.%of degree $\check{\lambda}_G\in \pi_1(G)$ and suppose that  $\pi_1(G)$ is torsion free then $\deg(F_G)=0$
\end{theorem}
\begin{proof}
Let $F_G$ be such a bundle. Let $j : F_\Gamma \to X$ be a finite bundle such that $F_G \cong F_\Gamma \times^{\Gamma} G$. Let $T$ be a maximal torus and $B\supset T$ a Borel containing $T$, and choose a reduction $F_B$ of $F_G$ to a Borel. %By \ref{trivialized} 
We know that $j^{*}F_G$ is trivial. Since 
\begin{equation}
    j^{*}F_B \times^{B}G=j^{*}(F_B \times^{B} G) = j^{*}F_G,
\end{equation} we see that $j^{*}F_B \times^{B} G$  is trivial. We have that $\pi_{0}(\calm_{B,F_{\Gamma}})=\pi_0(\calm_{T,F_\Gamma})=X_{*}(T)$ and this maps surjectively onto $\pi_0(\calm_{G,F_\Gamma})$. The fact that $j^{*}F_B$ maps to the trivial torsor means that it corresponds to $0$ in $\pi_1(G)=X_{*}(T)/\Phi^{\vee} = \pi_0(\calm_{G,F_\Gamma})$. This implies that the degree of $j^{*}F_B$, seen as an element in $X_{*}(T)$, %we have that $j^{*}F_B$ 
is a sum of coroots. The equality $\pi_0(\calm_{B})=\pi_0(\calm_T)$ is induced by the morphism $\pi_T : B\to T$, so $\pi_{T,*}j^{*}F_B$ also corresponds to a sum of coroots. Since $\pi_{T,*}j^{*}F_B = j^{*}\pi_{T,*}F_B$, the conclusion follows if we can show that the morphism
\begin{equation}
    j^{*} : \calm_{T,X} \to \calm_{T,F_\Gamma}
\end{equation}
has the property that, if $j^{*}F_T$ has degree in $\Phi^{\vee}$, then the same holds for a multiple of $\deg(F_T)$.

If $F_T$ corresponds to the cocharacter $\mu_{F_T}$, then $j^{*}F_T$ corresponds to the cocharacter $\mu_{j^{*}F_T}=\deg(j)\mu_{F_T}$. Thus if $\mu_{F_T} = \sum_{i=1}^{n} a_i \alpha_{i}^{\vee} + \mu$, where $\alpha_i$ are the simple roots and $\mu \in X_{*}\setminus \Phi^{\vee}$ then $$\mu_{j^{*}F_T} = \sum_{i=1}^{n}\deg(j)a_i \alpha_{i}^{\vee} + \deg(j)\mu = \sum_{i=1}^{n} a_i' \alpha_{i}^{\vee}$$ Hence, $\deg(j)\mu \in \Phi^{\vee}$.  %we are done. %Since $\pi_1(G)=X_{*}(T)/\{\text{coroots}\}$ was assumed torsion-free, we see that $\deg(j)\mu \in \{\text{coroots}\}$ must imply that $\mu \in \{\text{coroots}\}$ and thus $\mu_{F_T} \in \{\text{coroots}\}$.

We now apply this to our situation above, i.e., with $F_T \coloneqq \pi_{T,*}F_B$, and since $\pi_1(G)=X_{*}(T)/\Phi^{\vee}$ we can conclude that $\deg(F_G)$ is torsion.
\end{proof}
\begin{proposition}
Let $G$ be a connected, reductive group. %Suppose that $\Char(k)=0$ or that $\Char(k)\geq N$, for $N$ as in Theorem \ref{ef.via.reps}. 
If $X$ is an elliptic curve, then every essentially finite $G$-bundle over $X$ has degree 0.
\end{proposition}
\begin{proof}
We argue by induction on the dimension of $G$. If $\dim(G)=1$ then $G\cong \mg_m$ and the result follows since it is true for all vector bundles. Suppose now that $\dim(G)=n> 1$. Let $F_G$ be an essentially finite $G$-bundle of degree $d$. By \cite{dragos} there is a proper Levi $L$ and a degree $d'\in \pi_1(L)$ such that the inclusion $\iota\colon L\to G$ induces a surjection $\calm_{L,X}^{d'}\to \calm_{G,X}^{d}$. Let $F_L$ be a reduction of structure group of $F_G$ to $L$. Since $F_G$ is essentially finite there is a faithful representation $\rho\colon G\to \text{GL}_V$ such that $\rho_{*}F_G \cong (\rho\circ \iota)_{*}F_L$ is essentially finite. By Theorem \ref{ef.via.reps} this implies that $F_L$ is essentially finite. Since $L$ is a proper Levi, by induction $d'=0$, whence $d=0$.
\end{proof}
%\begin{remark}
%The same proof seems to suggest that for any reductive group $G$, if $F_G$ is essentially finite, then $\deg(F_G)=0$ or $\deg(F_G)$ is torsion. Hence, since $\pi_1(G)$ is a finitely generated abelian group, in the decomposition $\calm_G = \coprod_{\gamma \in \pi_1(G)}\calm_{G}^{\gamma}$, essentially finite torsors can only occur in finitely many of the components.
%\end{remark}
%
%
%
%
If the characteristic of $k$ is positive, there is a stronger notion of semistability, defined as follows. Let $\sigma_X : X \to X$ denote the absolute Frobenius of $X$.
\begin{definition}
A $G$-torsor $F_G$ is said to be {\bf strongly semistable} if for all $n>0$, $(\sigma_{X}^{n})^{*}F_G$ is semistable.
\end{definition}
\begin{proposition}\label{ef.strongly.semistable}
Every essentially finite $G$-torsor is strongly semistable.
\end{proposition}
\begin{proof}
For any algebraic group $H$, and any $H$-torsor, if $\sigma_H : H \to H$ denotes the absolute Frobenius of $H$, then we have that
\begin{equation}\label{abs.frob.push.pull}
    (\sigma_H)_{*}F_H \cong \sigma_{X}^{*}F_H.
\end{equation}
Let now $F_G$ be an essentially finite $G$-torsor. Let $j : F_\Gamma \to X$ be a finite bundle such that $F_G \cong F_\Gamma \times^{\Gamma} G$. Then by (\ref{abs.frob.push.pull}) applied to $\Gamma$ and since the push-forward along group morphisms commutes with pullbacks, we have that
\begin{equation}
    \iota_{*}(\sigma_{\Gamma,*}) F_{\Gamma} \cong \iota_{*} \sigma_{X}^{*}F_{\Gamma}
    \cong (\sigma_X)^{*}\iota_{*}F_{\Gamma} 
    \cong (\sigma_X)^{*}F_G.
\end{equation}
Hence $(\sigma_X)^{*}F_G$ is essentially finite and thus semistable. The statement follows similarly via induction.
\end{proof}
\subsection{The prestack of essentially finite torsors}
Let $\calm_{G}^{\ef}$ denote the functor
\begin{equation}
    \begin{aligned}
    \calm_{G}^{\ef} : \aff_{k}^{\op} &\to \grpds \\
    U &\mapsto \Big\{ \text{essentially finite $G$-torsors over $U\times X$} \Big \} + \Big\{ \text{isomorphism of $G$-torsors} \Big\}.
    \end{aligned}
\end{equation}
It is immediate that $\calm_{G}^{\ef}$ is a subfunctor of $\calm_G^{\Ss}$. 
\begin{proposition}
The functor $\calm_{G}^{\ef}$ is a $k$-prestack.
\end{proposition}
\begin{proof}
First suppose that $f : U'\to U$ is a morphism in $\aff_{k}^{\op}$ and suppose $F_G$ is an essentially finite $G$-torsor over $U\times X$. Let $(U_i \to U)$ be a cover and $(g_{ij} : g_{ij} \in G(U_{ij}))$ a cocycle for $F_G$. Then $(f^{*}U_i \to U')$ is a cover of $U'$ and $(f^{*}g_{ij})_{ij}$ is a cocycle for $f^{*}F_G$. Indeed, since $g_{ij}g_{jk}=g_{ik}$ we see that
\begin{equation}
    f^{*}g_{ij} f^{*}g_{jk} (x) = g_{ij}(f(x))g_{jk}(f(x))=g_{ik}(f(x))=f^{*}g_{ik}(x).
\end{equation}
The torsor $f^{*}F_G$ is also essentially finite since if $g_{ij} \in \Gamma(U_{ij})\subset G(U_{ij})$ for some finite group $\Gamma$, then $f^{*}g_{ij} = g_{ij}\circ f$ also takes values in $\Gamma$. Since $\calm_{G}^{\Ss}$ is a lax functor we see that $\calm_{G}^{\ef}$ is one as well.

Next it is clear that if $F_G,F_G' \in \calm_{G}^{\ef}(U)$, then $\underline{\text{Isom}}(F_G,F_G') : \aff_{/U} \to \set$ is a sheaf since homomorphisms of finite $G$-torsors are simply homomorphisms of $G$-torsors and $\calm_{G}^{\Ss}$ is a stack.

%Finally, we need to show that all descent data are effective. To this end, suppose $(U_i \to U)$ is a cover of $U\in \aff_{k}^{\op}$, $F_i \in \calm_{G}^{\ef}(U_i)$ and $\phi_{ij} : F_i|_{U_{ij}}\to F_j|_{U_{ij}}$ isomorphisms satisfying the cocycle condition. Then because $\calm_G$ is a $k$-stack we know that there is some $G$-torsor $F \in \calm_{G}^{\Ss}(U)$ and isomorphisms $F|_{U_i}\cong F_i$. Let $\Gamma$ denote the structure group of $F$, and let $\Gamma_i$ denote the structure group of $F_i$. Since $F|_{U_i}\cong F_i$ we know that $\Gamma \times U_i \cong \Gamma_i$. But $U_i\to U$ is faithfully flat and in particular surjective, hence $\Gamma \times U_i \twoheadrightarrow \Gamma$ so $\Gamma$ is finite since $\Gamma_i \cong \Gamma\times U_i$ is. Thus, $F\in \calm_{G}^{\ef}(U)$.
\end{proof}
\begin{remark}
Note however that $\calm_{G}^{\ef}$ is not a stack since the descent data is not necessarily effective. Indeed, let $G=\text{GL}_n$ and let $E$ be a vector bundle which is not essentially finite. Let further $(U_i \to X)$ be a trivialising cover of $E$, with trivilising morphisms $\phi_i : E|_{U_i} \rightarrow \calo_{U_i}^{n}$. Then $E|_{X\times U_i}$ with the morphisms $(\id \times \phi_{j}^{-1}) \circ (\id \times \phi_i)$ form a descent data for $E|_{X\times X}\in \calm_{G}(X)$. Now, if $E|_{X\times X}$ is essentially finite, then so is $E$. Indeed, by \cite{bds} we have a proper surjective morphism $f : Y \to X\times X$ such that $f^{*}E_{X\times X}$ is trivial, and by composing with the projection $X\times X \to X$ we have a proper surjective morphism $g : Y\to X$ such that $g^{*}E$ is trivial. Since $E$ was assumed not to be essentially finite, we conclude that $E|_{X\times X}$ is not essentially finite and the descent data constructed is not effective.
\end{remark}
The following statement is immediate, but will be important for us in the final section.
\begin{proposition}\label{productfinite}
Let $G$ and $G'$ be reductive groups. The isomorphism $\calm_{G\times G'}^{\Ss}\xrightarrow{\cong} \calm_{G}^{\Ss} \times \calm_{G'}^{\Ss}$ restricts to an isomorphism
\begin{equation}
    \calm_{G\times G'}^{\ef} \cong \calm_{G}^{\ef}\times \calm_{G'}^{\ef}.
\end{equation}
\end{proposition}
\begin{proof}
The isomorphism on objects is given by 
\begin{equation}
\begin{aligned}
    F_{G\times G'} &\mapsto (\pi_{*} F_{G\times G'}, \pi_{*}' F_{G\times G'}), \\
    (F_G , F_{G'})&\mapsto F_G \times F_{G'},
\end{aligned}
\end{equation}
where $\pi : G\times G' \to G$ and $\pi' : G\times G' \to G'$ are the projections. If $\Gamma \subset G\times G'$ is a finite structure group of $F_{G\times G'}$, then $\pi(\Gamma)$ and $\pi'(\Gamma')$ are evidently finite structure groups of $F_G$ and $F_{G'}$ respectively. Similarly, finite structure groups $\Gamma$ and $\Gamma'$ of $F_G$, respectively $F_{G'}$, give a finite structure group, $\Gamma \times \Gamma'$ of $F_G \times F_{G'}$.
%\textcolor{red}{Check that slope map respects products, then argue about semistability. Otherwise: one direction follows from HOFFMANN center to center etc. and the other from thesis ... in char 0.}
\end{proof}
%
%
%
%
%
% DENSITY OF TORSORS
%
\section{Density of essentially finite torsors}\label{section.density}
In this section we prove the density statements made in the introduction. The section is divided into subsections, depending on the genus of $X$.
\subsection{Preliminaries}
\begin{proposition}\label{dense.section}
Suppose $\pi : G\to H$ is a morphism of algebraic groups such that $\pi(Z(G)^{0})\subset Z(H)^{0}$. If $\pi$ admits a section $s : H\to G$ such that $s(Z(H)^{0})\subset Z(G)^{0}$, then density of $M_{G}^{\ef}$ in $M_{G}^{\Ss,0}$ implies density of $M_{H}^{\ef}$ in $M_{H}^{\Ss,0}$.
\end{proposition}
\begin{proof}
We prove the contrapositive. Thus, suppose that $M_{H}^{\ef}$ is not dense in $M_{H}^{\Ss,0}$.
Since $\pi_{*}$ takes essentially finite $G$-torsors to essentially finite $H$-torsors, by (\ref{deg0push}) we have a commutative diagram as follows
\begin{equation}
    \begin{tikzcd}
    M_{G}^{\ef} \arrow[r, bend left = 10, "\pi_{*}"] \arrow[d, hookrightarrow] & M_{H}^{\ef} \arrow[l, bend left = 10, "s_{*}"]  \arrow[d, hookrightarrow] \\
    M_{G}^{\Ss,0} \arrow[r, bend left = 10, "\pi_{*}"] & M_{H}^{\Ss,0} \arrow[l, bend left = 10, "s_{*}"]
    \end{tikzcd}
\end{equation}
Since  $\pi_{*}$ is continuous, $\pi_{*}\Big(\overline{M_{G}^{ef}}\Big)\subset \overline{M_{H}^{ef}}$. Suppose now on the contrary that $M_{G}^{\ef}$ is dense in $M_{G}^{\Ss,0}$. Pick any $F\in M_{H}^{\Ss,0}$. Then $s_{*}F\in M_{G}^{\Ss,0}= \overline{M_{G}^{\ef}}$. But since $\pi_{*}s_{*}=\id$ we see that
\begin{equation}
    F = \pi_{*}s_{*}F \in \pi_{*}\Big(\overline{M_{G}^{\ef}}\Big)\subset \overline{M_{H}^{\ef}},
\end{equation}
which implies that $\overline{M_{G}^{\ef}}=M_{H}^{\Ss,0}$. Contradiction.
\end{proof}
\begin{remark}
The condition on the centers is to make sure that the pushforward of a semistable bundle is semistable.
\end{remark}
\begin{corollary}\label{direct.products}
Let $G$ be a direct product of reductive groups $G_1$ and $G_2$. If $M_{G_i}^{\ef}$ is not dense in $M_{G_i}^{\Ss,0}$ for some $i=1,2$, then $M_{G}^{\ef}$ is not dense in $M_{G}^{\Ss,0}$.
\end{corollary}
\begin{proof}
We use the projection $\pi_i : G \to G_i$ and apply the previous proposition.
\end{proof}
%\textcolor{red}{Check if pushforward of degree 0 semistable is semistable. Then we can remove the hypothesis about the centers.}
\begin{proposition}\label{dense.torus}
Let $G=T$ be a torus. Then $M_{T}^{\ef}$ is dense in $M_{T}^{\Ss,0}$. 
\end{proposition}
\begin{proof}
First suppose $T=\mg_m$. Then $M_{T}^{\Ss,0}=\text{Jac}^0(X)$ is the Jacobian of $X$ and essentially finite $\mg_m$-torsors corresponds to finite line bundles which corresponds to torsion points on $\text{Jac}^0(X)$, which are dense. If $T\cong \mg_m^r$ for $r>1$, then we apply Proposition \ref{productfinite}  and the statement follows. %(\textcolor{red}{More to check: degree 0 times degree 0 to degree 0}).
\end{proof}
\subsection{Genus 0}
Let now $X=\mathbb{P}_{k}^1$, where $k$ is an arbitrary algebraically closed field. %It is well-known that $M_{G}^{\Ss,0}(k)$ is a singleton so the density statement is immediate.
%Over $\mathbb{C}$, by work of \cite{ramana1}, every semistable $G$-torsor of degree 0 over a smooth projective curve arises from a representation of the topological fundamental group. Since $\pi_1^{\text{top}}(X)=1$, it is immediate that only the trivial essentially finite $G$-torsor exist. In characteristic $p>0$, for any semisimple group $G$, or $G=GL_n$, and for any strongly semistable $G$-torsor, $F_G$, Biswas, Parameswaran and Subramanian \cite{BiswasStrongly} have constructed a monodromy group, $M$, and a reduction of structure group of $F_G$ to $M$, where $M$ is a quotient of the Nori fundamental group, $\pi_1^{N}(X)$. Since $\pi_1^{N}(X)=1$ \cite{Nori2}, we see via Proposition \ref{ef.strongly.semistable} that any essentially finite $G$-torsor is trivial. 
%\todo{Add something about $M_{G}^{\Ss}=\{\ast \}$.}
 By Proposition \ref{nori.ef.equiv} we immediately have the following statement.
\begin{proposition}\label{ef.over.p1}
Every essentially finite $G$-bundle over $X$ is trivial.
\end{proposition}
\begin{proof}
Since $\pi_{1}^{N}(X,x)$ is trival, the statement follows from Proposition \ref{nori.ef.equiv}.
\end{proof}
It is also well-known that $M_{G}^{\Ss,0}(k)$ is a singleton so the density statement is immediate.

For the remainder of this section, we give a different proof of Proposition \ref{ef.over.p1}, which might be interesting in its own right. We do this by using the Tannakian interpretation of essentially finite $G$-bundles and the classification of $G$-bundles on X. %, we give a self-contained proof that any essentially finite $G$-torsor is trivial, which might be interesting in its own right. We do this by using the Tannakian interpretation of the classification of $G$-bundles on $X$.
% In characteristic $p>0$, for any $G$-torsor, $F_G$‚ Biswas, Parameswaran and Subramanian \todo{add ref} have constructed a monodromy group, $M$, and a reduction of structure group of $F_G$ to $M$, where $M$ is a quotient of the Nori fundamental group, $\pi_1^{N}(X)$. Since $\pi_1^{N}(X)=1$, we see that any such $G$-torsor is trivial. Here we give a new proof that any essentially finite $G$-torsor is trivial by using the classification of $G$-torsors on $X$ by Grothendieck and Harder.

%More precisely, 
The classification of $G$-bundles on $X$ was initially done by Grothendieck \cite{Gro57} and  by Harder  \cite{hard}  for characterstic $p$. In \cite{torsorsprojective} Anschütz gives a Tannakian interpretation of this classification. We thus begin by introducing the relevant notions from \cite{torsorsprojective}.

Over $X$ there is a canonical $\mg_m$-torsor 
\begin{equation}
\begin{aligned}
    \eta : \mathbb{A}^2\setminus \{ 0 \} &\to X \\
    (x_0,x_1) &\mapsto [x_0 : x_1],
\end{aligned}
\end{equation} 
often called the Hopf bundle. Pushforward along this bundle defines an exact, faithful tensor functor
\begin{equation}
    \begin{aligned}
    \mathcal{E} : \rep_k(\mg_m) &\to \vect_X \\
    V &\mapsto \mathbb{A}^2\setminus \{ 0 \} \times^{\mg_m} V.
    \end{aligned}
\end{equation}
%which induces an isomorphism onto its essential image. 
Taking the Harder-Narashiman filtration of a vector bundle over $X$ defines a fully faithful tensor functor
\begin{equation}
    HN : \vect_X \to \text{FilVec}_X
\end{equation}
from $\vect_X$ to the category of filtered vector bundles. Finally we can take the graded pieces of a filtered vector bundle and this defines an exact tensor functor
\begin{equation}
    \gr : \text{FilVec}_X \to \text{Grvec}_X,
\end{equation}
where $\text{GrVec}_X$ is the category of graded vector bundles. 
\begin{proposition}
[Anschütz, \cite{torsorsprojective}, Lemma 2.3] The composition 
\begin{equation}
  \cale_{\gr} : \rep_k(\mg_m) \xrightarrow{\mathcal{E}} \vect_X \xrightarrow{HN} \text{FilVec}_X \xrightarrow{\gr} \text{GrVec}_X 
\end{equation}
is an equivalence of tensor categories onto its essential image, which consists of graded bundles $E = \bigoplus_{n\in \mathbb{Z}} E_{i}$ such that each $E_i$ is semistable of slope $i$.
\end{proposition}
The main Theorem of Grothendieck, restated in the Tannaka language by Anschütz is now given by
\begin{proposition}
[Anschütz, \cite{torsorsprojective}, Theorem 3.3] Let $G$ be a reductive group over $k$. The composition with $\cale$ defines a faithful functor
\begin{equation}
  \Phi :  \underline{\Hom}^{\otimes}(\rep_k(G), \rep_k(\mg_m)) \to \underline{\Hom}^{\otimes}(\rep_k(G),\vect_X),
\end{equation}
which induces a bijection
\begin{equation}
    \Hom^{\otimes}(\rep_k(G),\rep_k(\mg_m)) \cong H_{\text{ét}}^{1}(X,G)
\end{equation}
on isomorphism classes.
\end{proposition}
The inverse of this is given by composition with $\cale_{\gr}^{-1} \circ \gr \circ \text{HN}$. Using this we can now describe all essentially finite $G$-bundles on $X$.
\begin{proposition}
Every essentially finite $G$-torsor over $X$ is trivial.
\end{proposition}
\begin{proof}
Let $F_G : \rep_k(G) \to \vect_X$ be an essentially finite torsor. By Proposition (\ref{ef.tannaka}) there exists a commutative diagram of tensor functors
\begin{equation}
    \begin{tikzcd}
    \rep_k(G) \arrow[r, "F_G"] \arrow[d, "\alpha"] & \vect_X \\
        \rep_k(\Gamma) \arrow[ru, "F_\Gamma"] & 
    \end{tikzcd}
\end{equation}
for some finite group $\Gamma$. By \cite{torsorsprojective} this sits inside the following larger diagram
\begin{equation}
    \begin{tikzcd}
    \rep_k(G) \arrow[d, "\alpha"'] \arrow[r, "\Phi^{-1}(F_G)"] \arrow[rr, bend left = 30, "F_G"] & \rep_k(\mg_m) \arrow[r, "\mathscr{E}"] & \vect_X \arrow[r, "\text{HN}"] & \text{Fil}\vect_X \arrow[r, "\text{gr}"] & \text{Gr}\vect_X \arrow[lll, bend left = 30, "\mathscr{E}_{\text{gr}}^{-1}"] \\
    \rep_k(\Gamma) \arrow[ru, dotted, "f"] \arrow[rru, "F_{\Gamma}"'] & & & & 
    \end{tikzcd}
\end{equation}
where $f$ is defined to be the composition
\begin{equation}
    f\coloneqq \mathscr{E}_{\text{gr}}^{-1} \circ \text{gr} \circ \text{HN} \circ F_\Gamma.
\end{equation}
Since all functors are tensor functors, so is $f$. By \cite{Tannaka} $f$ is induced by a morphism 
\begin{equation}
\tilde{f} : \mg_m \to \Gamma.
\end{equation}
Since $\mg_m$ is connected and $\Gamma$ is discrete we see that $\tilde{f}$ and thus $f$ is the trivial map. But this implies that
\begin{equation}
    F_G \cong \mathscr{E}\circ \Phi^{-1}(F_G) \cong \mathscr{E} \circ \mathscr{E}_{\text{gr}}^{-1}\circ \text{gr}\circ \text{HN} \circ F_G \cong \mathscr{E} \circ \mathscr{E}_{\text{gr}}^{-1}\circ \text{gr}\circ \text{HN} \circ F_\Gamma \circ \alpha \cong  \mathscr{E} \circ f \circ \alpha
\end{equation}
is the trivial torsor.
\end{proof}
%
%
%
%
%  GENUS 1
%
\subsection{Genus 1}
In the case when $X$ is an elliptic curve, the density result follows almost immediately from known properties of $M_{G}^{\Ss}$, studied by Laszlo \cite{laszlo} in characteristic 0 and Fr{\u{a}}{\c{t}}il{\u{a}} in charactierstic $p$ \cite{dragos}. %Let us thus quickly recall the main result of \cite{dragos} before stating the density result for $X$.
%
%The main result of \cite{dragos} is the following.
%\begin{proposition}[Fratila, \cite{dragos}], Theorem 1.1] Let $\check{\lambda}_G \in \pi_1(G)$. There exists an up to conjugation unique Levi, $L\subset G$, and a degree $\check{\lambda}_L \in \pi_1(L)$ such that the following hold.
%\begin{enumerate}\label{dragos.prop}
%    \item The inclusion $L\hookrightarrow G$ induces a morphism $M_{L}^{\check{\lambda}_L}\to M_{G}^{\check{\lambda}_G}$ and all semistable $L$-torsors in $M_{L}^{\check{\lambda}_L}$ are stable.
%    \item The induced map $M_{L}^{\check{\lambda}_L}\to M_{G}^{\check{\lambda}_G}$ is finite, generically Galois with Galois group $W_{L,G}=N_{G}(L)/L$.
 %   \item The morphism above induces an isomorphism
 %   \begin{equation}
 %       M_{L}^{\check{\lambda}_L}/W_{L,G}\xrightarrow{\cong} M_{G}^{\check{\lambda}_G}.
 %   \end{equation}
%\end{enumerate}
%\end{proposition}
%In the case that $\check{\lambda}_G=0$, then $L=T$ is a maximal torus and $\check{\lambda}_L = 0$.
%
%With this result, we immediately obtain the following.
\begin{proposition}
Suppose $X$ is an elliptic curve. Then $M_{G}^{\ef}$ is dense in $M_{G}^{\Ss,0}$ for any reductive group $G$.
\end{proposition}
\begin{proof}
Let $T$ be a maximal torus of $G$ and let $W$ be the corresponding Weyl group. Then, by \cite[Theorem 4.16]{laszlo} and \cite[Theorem 1.1]{dragos}, we have an isomorphism
\begin{equation}
    \varphi : M_{T}^{\Ss,0}/W \to M_{G}^{\Ss,0}
\end{equation}
induced by the inclusion $\iota : T \hookrightarrow G$. Since $\iota_{*}(M_{T}^{\ef})\subset M_{G}^{\ef}$, the result follows from Proposition \ref{dense.torus}.
\end{proof}
%
%
%
%
%
%
% GROUPS OF SEMISIMPLE RANK 1
%
\subsection{Genus $g\geq 2$}
Let now $X$ be of genus $g \geq 2$. 
Suppose first that $\Char(k)=p>0$ and let $\sigma_X$ denote the absolute Frobenius of $X$. Then a vector bundle $E$ is called periodic under the action of Frobenius if $E\cong (\sigma_{X}^{n})^{*}E$ for some integer $n\geq 1$. If $E$ is such a vector bundle, then, we know that $E$ is trivialized by an étale cover \cite[Theorem 1.1]{lang.stuhler}. Hence, $E$ is essentially finite \cite[Theorem 1]{bds}. In \cite[Proposition 4.1 and corollary 5.1]{frobdensity} the authors proved that, for any $n>0$, the set of $k$-points in $M_{\text{GL}_n}^{\Ss,0}$ (resp $M_{\text{SL}_n}^{\Ss}$) periodic under the action of Frobenius is dense. Hence, the set of $k$-points corresponding to essentially finite vector bundles is also dense. Hence, we may state the following.
\begin{proposition}
Let $k$ be of characteristic $p>0$. For any $n>1$, $M_{\text{PGL}_n}^{\ef,0}$ is dense in $M_{\text{PGL}_n}^{\Ss,0}$.
\end{proposition}
\begin{proof}
This follows from the previous discussion and the fact that the projection $\text{GL}_n\to \text{PGL}_n$ induces a surjection $M_{\text{GL}_n}^{\Ss,0}\to M_{\text{PGL}_n}^{\Ss,0}$ (see \cite[Proposition 18]{serre}) which takes essentially finite $\text{GL}_n$-bundles to essentially finite $\text{PGL}_n$-bundles.
\end{proof}
Let now $k$ be of characteristic zero. We restrict ourselves to split reductive groups of semisimple rank 1. By classical results (see e.g., \cite[Chapter 21]{milne}) these are all given by the following list.
\begin{proposition}
Let $G$ be a split reductive group of semisimple rank 1. Then, up to isomorphism, $G$ is one of the following groups:
\begin{equation}
    \text{GL}_2\times \mg_{m}^{r}, \hspace{2cm} \text{SL}_2\times \mg_{m}^{r}, \hspace{2cm} \text{PGL}_2 \times \mg_{m}^{r}, \hspace{2cm} r\in \mathbb{N}.
\end{equation}
\end{proposition}
Hence, by Proposition \ref{dense.section} applied to the projection map, if we show non-density for $\text{SL}_2$, $\text{GL}_2$, and $PGL_2$, we show it for all split reductive groups of semisimple rank 1. Now, by known results (\cite[I.3 page 7]{dominant}), the quotient maps on the respective groups induce dominant morphisms
\begin{equation}
\begin{aligned}
    M_{\text{SL}_2}^{\Ss} &\to M_{\text{PGL}_2}^{\Ss,0} \\
    M_{\text{GL}_2}^{\Ss,0} &\to M_{\text{PGL}_2}^{\Ss,0}.
\end{aligned}
\end{equation}
Thus, to show non-density for split reductive groups of semisimple rank 1 it suffices to show it for $\text{PGL}_2$, which we do now.
%LEMMA dimM_O(2)

To do this we need a bound on the dimension of $M_{\text{O}(2)}^{\Ss}$. For a connected reductive group $G$ it is well-known that $\dim \calm_G = \dim(G)(g-1)$ (see e.g. \cite{sorger}). Since $\text{O}(2)$ is not connected, we compute $\dim \calm_{\text{O}(2)}$ following the approach for connected reductive groups.
\begin{lemma}\label{dim.m.o2}
We have that $\dim \calm_{\text{O}(2)}=g-1$. %and $\dim M_{O(2)}^{\Ss} \leq g$.
\end{lemma}
\begin{proof}
Let $F_{\text{O}(2)}$ be an $\text{O}(2)$ bundle and let $\mathfrak{o}_2$ denote the Lie algebra of $\text{O}(2)$. Let further $\text{Ad} : \text{O}(2) \to \text{GL}(\mathfrak{o}_2)$ denote the adjoint representation and let $E\coloneqq \text{Ad}_{*}F_{\text{O}(2)}$. By definition we know that the dimension of $\calm_{\text{O}(2)}$ at the point $F_{\text{O}(2)}$ is the rank of the cotangent complex at $F_{\text{O}(2)}$, which is equal to $-\chi(X,E)$. By Riemann-Roch we thus have that
\begin{equation}\label{dim.o2}
\begin{aligned}
    \dim \calm_{\text{O}(2)} &= -\deg(E) - \rk(E)\chi(X,\calo_X) \\
    &=-\deg(E) + g-1.
\end{aligned}
\end{equation}
%We now show that the adjoint representation is self dual. To this end, 
By identifying $\text{O}(2)$ as the matrices
\begin{equation}
    \text{O}(2) = T' \coprod T' \{ \begin{bmatrix} 0 & 1 \\ 1 & 0 \end{bmatrix} \}, \,\,\, T' = \{ \begin{bmatrix} t & 0 \\ 0 & t^{-1} \end{bmatrix} : t \in \mg_m \},
\end{equation}
one sees immediately that the adjoint representation is self dual.
%Then we see that $\mathfrak{o}_2 \cong \{ \begin{bmatrix} a & 0 \\ 0 & -a \end{bmatrix} : a \in k \}$ and for $M \in \text{O}(2)$ we have that
%\begin{equation}
%    \text{Ad}(M)\begin{bmatrix} a & 0 \\ 0 & -a \end{bmatrix}= M\begin{bmatrix} a & 0 \\ 0 & -a \end{bmatrix}M^{-1} = \begin{cases} \begin{bmatrix} a & 0 \\ 0 & -a \end{bmatrix} & M \in T' \\
%    \begin{bmatrix} -a & 0 \\ 0 & a \end{bmatrix} & M \in T' \{ \begin{bmatrix} 0 & 1 \\ 1 & 0 \end{bmatrix} \}
%    \end{cases}.
%\end{equation}
%If $e_1 = \begin{bmatrix} 1 & 0 \\ 0 & -1 \end{bmatrix}$ denotes the standard basis of $\mathfrak{o}_2$, and $e_{1}^{\vee} : \begin{bmatrix} a & 0 \\ 0 & -a \end{bmatrix} \mapsto a$, the corresponding basis for $\mathfrak{o}_{2}^{\vee}$, then the action of $O(2)$ on $\mathfrak{o}_{2}^{\vee}$ is given by
%\begin{equation}
%    M \cdot ae_{1}^{\vee} = \begin{cases} ae_{1}^{\vee} & M \in T' \\
%    -ae_{1}^{\vee} & M \in T' \{ \begin{bmatrix} 0 & 1 \\ 1 & 0 \end{bmatrix} \}
%    \end{cases}.
%\end{equation}
%We thus see that $\mathfrak{o}_2$ and $\mathfrak{o}_{2}^{\vee}$ are isomorphic as $\text{O}(2)$-modules. 
%
Hence, $E \cong E^{\vee}$ and thus $\deg(E)=-\deg(E)$ whence $\deg(E)=0$. From equation (\ref{dim.o2}) we conclude that $\dim \calm_{\text{O}(2)}=g-1$. %For $M_{O(2)}^{\Ss}$ simply take an $O(2)$-bundle $F_{O(2)}$ which is stable as a vector bundle, and note that $\dim \Aut(F_{O(2)})\leq \dim Z(GL_2)=1$.
\end{proof}
\begin{lemma}\label{ss.to.ss}
Let $\iota$ denote an inclusion $\iota\colon \text{O}(2) \hookrightarrow \text{PGL}_2$. If $F_{O(2)}$ is a semistable  $\text{O}(2)$-bundle then $\iota_{*}F_{O(2)}$ is a semistable $\text{PGL}_2$-bundle.
\end{lemma}
\begin{proof}
The proof of \cite[Proposition 2.6]{balaji} applies verbatim, since an $\text{O}(2)$-bundle $F_{\text{O}(2)}$ is semistable if and only if $\iota_{*}'F_{\text{O}(2)}$ is semistable, where $\iota' \colon \text{O}(2)\hookrightarrow \text{GL}_2$ is the standard representation.
%Let $F_{O(2)}$ be an $O(2)$–bundle and let $F_{PGL_2}=\iota_{*}F_{O(2)}$. Let $\rho \colon PGL_2 \to GL(W)$ be a representation of $PGL_2$. Since $PGL_2$ is semistable, it suffices to show that $E\coloneqq \rho_{*}F_{PGL_2}$ is semistable. Let $V$ denote the standard representation of $O(2)$, i.e., $i \colon O(2)\hookrightarrow GL_2\cong GL(V)$. Since $V$ is faithful, it generates $\text{Rep}_k(O(2))$. Thus, the $O(2)$-module $W$ is a subquotient of $\oplus_{(a_i,b_i)}T^{a_i,b_i}(V)$, for some $a_i,b_i \in \mathbb{N}$. This implies that $E$ is a subquotioent of $\oplus_{(a_i,b_i)}T^{a_i,b_i}(i_{*}F_{O(2)})$. Since $i_{*}F_{O(2)}$ is semistable of degree 0, so is $\oplus_{(a_i,b_i)}T^{a_i,b_i}(i_{*}F_{O(2)})$. Finally, by Lemma \todo{add ref} $E$ is degree 0 whence $E$ is semistable, as it is a degree 0 subbundle of a semistable vector bundle.\todo{Double-check: why is $E$ semistabe?}
\end{proof}
\begin{proposition}\label{pgl2.main.theorem} 
The subset of essentially finite $\text{PGL}_2$-torsors is not dense inside $M_{\text{PGL}_2}^{\Ss,0}$.
\end{proposition}
\begin{proof}
By \cite{finitesubgroup} the finite subgroups of $\text{PGL}_2$ are given by $\text{S}_4$, $\text{A}_5$, $\text{A}_4$ and for all $n\in \mathbb{N}$, $\mu_n$ and $\text{D}_n$. Furthermore, for each finite subgroup there is only one conjugacy class by Proposition 4.1 in \cite{beauville}. Hence, for a given finite subgroup $\Gamma$, we may choose any embedding $\iota : \Gamma \hookrightarrow \text{PGL}_2$ and unambiguously
 consider $\iota_{*}\calm_{\Gamma}\subset \calm_{\text{PGL}_2}^{\Ss,0}$. 

Now, for any such group $\Gamma$, $\iota_{*}M_{\Gamma}\subset M_{\text{PGL}_2}^{\Ss,0}$ is a finite number of points. Indeed, we have that
\begin{equation}
    H_{\text{et}}^{1}(X,\Gamma) = \Hom(\pi_{1}(X), \Gamma)
\end{equation}
and since $\pi_{1}(X)$ is (pro)finitely generated, we see that $H_{\text{et}}^{1}(X,\Gamma)$ is a finite set. Hence, to prove the proposition it is enough to show that the essentially finite torsors whose finite group is isomorphic to $\text{D}_n$ or $\mu_n$ for some $n>0$, is not dense. By abuse of notation, we still denote this subset by $M_{\text{PGL}_2}^{\ef,0}$.

Let $\pi : \text{GL}_2 \to \text{PGL}_2$ denote the quotient morphism. From \cite{finitesubgroup} Section 2 we thus see that we may choose the embedding such that for every such $\Gamma$, we have a commutative diagram
\begin{equation}\label{gamma.o2.pgl2}
    \begin{tikzcd}
    \Gamma \arrow[r, hookrightarrow] \arrow[rr, hookrightarrow, bend right = 30, "\iota"] & \pi(\text{O}(2)) \arrow[r, hookrightarrow] & \text{PGL}_2 
    \end{tikzcd},
\end{equation}
where $\text{O}(2)\subset \text{GL}_2$ is realized as the matrices
\begin{equation}
    \text{O}(2) = T' \coprod T' \{ \begin{bmatrix} 0 & 1 \\ 1 & 0 \end{bmatrix} \}, \,\,\, T' = \{ \begin{bmatrix} a & 0 \\ 0 & a^{-1} \end{bmatrix} : a \in \mg_m \}.
\end{equation}
Since $\pi(\text{O}(2))\cong \text{O}(2)$, and since $\iota'\colon O(2)\cong \pi(\text{O}(2))\hookrightarrow \text{PGL}_2$ is a closed embedding, the induced morphism $\iota_{*}':\calm_{\text{O}(2)}\to \calm_{\text{PGL}_2}$ is locally of finite type (see e.g., \cite[Fact 2.3]{hoffman}). By Lemma \ref{ss.to.ss} this induces a map $\iota_{*}':\calm_{\text{O}(2)}^{\Ss}\to \calm_{\text{PGL}_2}^{\Ss}$, which induces by the universal property of the coarse moduli space a morphism of finite type schemes $M_{\text{O}(2)}^{\Ss}\to M_{\text{PGL}_2}^{\Ss}$. By taking base change along $M_{\text{PGL}_2}^{\Ss,0}$ we obtain an open subscheme $U\subset M_{\text{O}(2)}^{\Ss}$ and a morphism of finite type $f \colon U\to M_{\text{PGL}_2}^{\Ss,0}$.
We thus obtain a Cartesian diagram
\begin{equation}
    \begin{tikzcd}
    U \arrow[r, "f"] \arrow[d, hookrightarrow] & M_{\text{PGL}_2}^{\Ss,0} \arrow[d, hookrightarrow] \\
    M_{\text{O}(2)}^{\Ss} \arrow[r, "\iota_{*}'"] & M_{\text{PGL}_2}^{\Ss}
    \end{tikzcd}
\end{equation}

Now, for any essentially finite $\text{PGL}_2$-torsor, $F_{\text{PGL}_2}$, by (\ref{gamma.o2.pgl2}) we may assume that $F_{\text{PGL}_2}=\iota_{*}'F_{\text{O}(2)}$ where $F_{\text{O}(2)}$ is an essentially finite $\text{O}(2)$-torsor.
Hence, we have a finite type morphism $f : U\to M_{\text{PGL}_2}^{\Ss,0}$ of projective varieties such that
\begin{equation}
    M_{\text{PGL}_2}^{\ef,0} \subset f(U).
\end{equation}
Thus, it suffices to show that $f$ is not dominant. Suppose it was. Then we obtain an inclusion of functions fields
\begin{equation}
    k(M_{\text{PGL}_2}^{\Ss,0}) \hookrightarrow k(U).
\end{equation}
This implies that
\begin{equation}
    3g-3= \dim M_{\text{PGL}_2}^{\Ss,0}= \text{tr.deg}_k k(M_{\text{PGL}_2}^{\Ss,0}) \leq \text{tr.deg}_k k(U) =\dim U = \dim M_{\text{O}(2)}^{\Ss} \leq g-1,
\end{equation}
where the last inequality follows from Lemma \ref{dim.m.o2}.
\end{proof}
From the statement for $\text{PGL}_2$ we obtain the same statement for $\text{SL}_2$.
\begin{corollary}\label{sl2}
The subset of essentially finite $\text{SL}_2$-torsors is not dense inside $M_{\text{SL}_2}^{\Ss,0}$.
\end{corollary}
\begin{proof}
Since the map $M_{\text{SL}_2}^{\Ss} \to M_{\text{PGL}_2}^{\Ss,0}$ is dominant this follows from Proposition (\ref{pgl2.main.theorem}).
\end{proof}
From this we obtain the same statement for $\text{GL}_2$.
\begin{corollary}
The subset of essentially finite $\text{GL}_2$-torsors is not dense inside $M_{\text{GL}_2}^{\Ss,0}$.
\end{corollary}
\begin{proof}
The same proof as above applies, or we have the following.
Consider the map
\begin{equation}
    \det : M_{\text{GL}_2}^{\Ss,0} \to \text{Jac}^0(X).
\end{equation}
Since $\det^{-1}(\calo_X) = M_{\text{SL}_2}^{\Ss}$ by Corollary (\ref{sl2}) we obtain the desired result.
\end{proof}
Finally, the complete statement is the following.
\begin{corollary} \label{dense.ssrk1}
For any split reductive group $G$‚ of semi-simple rank 1, the essentially finite $G$-torsors are not dense in $M_{G}^{\Ss,0}$.
\end{corollary}
\begin{proof}
This follows from the classification of split reductive groups and Proposition \ref{pgl2.main.theorem}.
\end{proof}
%We wrap this text with the following immediate corollary.
%\begin{corollary}
%If $G$ is a direct product, one of whose factor is a of semisimple rank 1, then $M_{G}^{\ef}$ is not dense in $M_{G}^{\Ss,0}$.
%\end{corollary}
%\begin{proof}
%This is Corollary \ref{direct.products} applied to this case.
%\end{proof}
%\begin{corollary}
%Suppose $G$ is a reductive group whose adjoint group has a group of semisimple rank 1 as a simple factor, then $\calm_{G}^{\ef}$ is not dense in $\calm_{G}^{\Ss,0}$.
%\end{corollary}
%\begin{proof}
%From the short exact sequence
%\begin{equation}
   % 1\to Z(G) \to G\to G^{\text{ad}} \to 1
%\end{equation}
%we obtain from Corollary 3.4 in \cite{hoffman} a surjection $\pi_{*} : \calm_G\to %\calm_{G^{\text{ad}}}$. By (ADD REF TO SCHIEDER and FINITE TO FINITE ABOVE) this induces a surjection
%\begin{equation}
  %  \pi_{*} : \calm_{G}^{\Ss,0}\to \calm_{G^{\text{ad}}}^{\Ss,0}.
%\end{equation}
%Since $G^{\text{ad}}$ is the product of its simple factors we see from above that $\calm_{G^{\text{ad}}}^{\ef}$ is not dense in $\calm_{G^{\text{ad}}}^{\Ss,0}$. Now, using the surjectivity of $\pi_{*}$ and the same argument as in Proposition (\ref{sl2}) we see that $\calm_{G}^{\ef}$ is not dense in $\calm_{G}^{\Ss,0}$.
%\end{proof}
\newpage
\printbibliography
\end{document}